\documentclass{amsart}
\usepackage{amssymb}
\usepackage{stmaryrd} 
\usepackage{amsmath} 
\usepackage{amscd}
\usepackage{amsbsy}
\usepackage{commath}
\usepackage{comment, enumerate}
\usepackage{geometry}
\usepackage[matrix,arrow]{xy}
\usepackage{hyperref}
\usepackage{tabularx}

\usepackage{mathrsfs}
\usepackage{color}
\usepackage{mathtools,caption}
\usepackage{tikz-cd}
\usepackage{longtable}
\usepackage[utf8]{inputenc}
\usepackage[OT2,T1]{fontenc}
\usepackage[normalem]{ulem}
\usepackage{hyperref}
\hypersetup{
	colorlinks=true,
	linkcolor=blue,
	filecolor=violet,
	citecolor=orange,
	urlcolor=purple,
	pdftitle={The Iwasawa mu-invariants of elliptic curves over Q},
}

\newtheorem{thm}{Theorem}[section]
\newtheorem*{thm*}{Theorem}
\newtheorem{dfn}[thm]{Definition}
\newtheorem{lem}[thm]{Lemma}
\newtheorem*{lem*}{Lemma}
\newtheorem*{sublem*}{Sublemma}

\newtheorem{prop}[thm]{Proposition}

\newtheorem{conj}[thm]{Conjecture}

\theoremstyle{definition}

\def\Z{\mathbf{Z}}
\def\F{\mathbf{F}}

\def\Q{\mathbf{Q}}
\def\C{\mathbf{C}}

\def\a{\alpha}
\def\b{\beta}

\def\s{\sigma}
\def\z{\zeta}

\def\Qcyc{\Q_{\text{cyc}}}

\def\Zpt{\Z_p\llbracket T \rrbracket}
\def\Gal{\text{Gal}}

\DeclareFontFamily{U}{wncy}{}
\DeclareFontShape{U}{wncy}{m}{n}{<->wncyr10}{}
\DeclareSymbolFont{mcy}{U}{wncy}{m}{n}
\DeclareMathSymbol{\Sha}{\mathord}{mcy}{"58}

\def\sse{\subseteq}
\title{The Iwasawa $\mu$-invariants of Elliptic Curves over $\Q$}

\author[]{Adithya Chakravarthy}
\date{}
\begin{document}

	\maketitle

        \begin{abstract}

     In this paper, we discuss a longstanding conjecture of Greenberg in the Iwasawa theory of elliptic curves. Greenberg's conjecture states that if $E/\Q$ is an elliptic curve with good ordinary reduction at $p$, and $E[p]$ is irreducible as a Galois module, then the Selmer group of $E$ over the cyclotomic $\Z_p$ extension of $\Q$ has $\mu$-invariant zero. We prove that if $E$ is an elliptic curve over $\Q$, then we have $\mu \leq 1$ for all but finitely many primes $p$ of good ordinary reduction.    
	\end{abstract}

    \section{Introduction}

    Classical Iwasawa theory studies how class numbers grow in infinite towers of number fields. Suppose we have a tower of number fields 
    \[ K = K_0 \sse K_1 \sse  K_2 \sse \dots \sse K_n \sse \dots  \]
    such that $\Gal(K_n/K)$ is cyclic of order $p^n$ for each $n \geq 1$. Put $K_{\infty} = \cup_{n=1}^{\infty} K_n$. Then $K_{\infty}/K$ is a $\Z_p$-extension: a Galois extension with Galois group isomorphic to the additive group $\Z_p$ of $p$-adic integers. Iwasawa proved the following fundamental theorem about the growth of class numbers in these towers:
    
	\begin{thm*}[Iwasawa]
		Let $K$ be a number field and let $K_{\infty}/K$ be a $\Z_p$ extension with intermediate layers $K_n$.  Suppose that $p^{e_n}$ is the exact power of $p$ dividing the class number of $K_n$. There are integers $\mu, \lambda, \nu$  such that 
		\[ e_n = \mu p^n + \lambda n + \nu \]
		for all $n$ sufficiently large.
	\end{thm*}	
	Given any number field $K$, there is an important example of a $\Z_p$-extension of $K$ called the \textit{cyclotomic $\Z_p$-extension} of $K$, denoted $K^{\text{cyc}}/K$. It is defined by letting $K^{\text{cyc}}$ be the appropriate subfield of $\cup_{n \geqq 1} K(\z_{p^n})$. Iwasawa made the following conjecture about the cyclotomic $\Z_p$-extension.
	\begin{conj}[Iwasawa's $\mu=0$ Conjecture] Let $K$ be a number field and let $K^{\text{cyc}}$ be the cyclotomic $\Z_p$ extension of $K$. Then $\mu=0$. 
	\end{conj}
	This conjecture is probably the deepest open problem from classical Iwasawa theory. Ferrero and Washington \cite{fw} have proven this conjecture in the case when $K/\Q$ is abelian. Other than this, Iwasawa's conjecture remains wide open.
	
	This paper is about a generalization of the Iwasawa $\mu=0$ conjecture to the Iwasawa theory of \textit{elliptic curves}. We now briefly outline the main aspects of the Iwasawa theory of elliptic curves and then state the analogue of Iwasawa's $\mu=0$ conjecture in this context. In a landmark paper \cite{mazur}, Mazur observed that many techniques from classical Iwasawa theory could be used to study elliptic curves. If $E$ is an elliptic curve over a number field $K$, let $\Sha(E/K)[p^{\infty}]$ be the $p$-primary part of the Tate-Shafarevich group of $E$ over $K$. We then have the following analogue \cite[Theorem 1.10]{greenberg} of Iwasawa's theorem.
	\begin{thm*}
		Let $\Q^{\text{cyc}} = \cup \Q_n$ denote the cyclotomic $\Z_p$ extension of $\Q$. Let $E/\Q$ be an elliptic curve, and let $p$ be a prime of good ordinary reduction. Assume that $\Sha(E/\Q_n)[p^{\infty}]$ is finite for all $n$. Let $p^{e_n}$ be the exact power of $p$ dividing $|\Sha(E/\Q_n)[p^{\infty}]|$. Then there exist integers $\mu, \lambda, \nu$  such that 
		\[ e_n = \mu p^n + \lambda n + \nu \]
		for all sufficiently large values of $n$.
	\end{thm*} 
	The quantity $\mu $ in the above formula is called the \emph{$\mu$-invariant} of $E$ over $\Q^{\text{cyc}}$. (See \hyperref[sec:preliminaries]{Section 2} for a more abstract definition of $\mu$.) In analogy with Iwasawa's $\mu=0$ conjecture, one might expect that $\mu =0$. It turns out that this is \textit{false}. Indeed, Mazur \cite{mazur} discovered that the elliptic curve $X_0(11)$ has a nonzero $\mu$-invariant for the prime $p=5$. However, Greenberg has conjectured \cite[Conjecture 1.11]{greenberg} that one can rescue Iwasawa's conjecture in the following sense: 
	
	\begin{conj}[Greenberg's $\mu=0$ Conjecture] 
		\label{greenberg}
		Let $E/\Q$ be an elliptic curve, and let $p$ be an odd prime of good ordinary reduction. If $E[p]$ is \textit{irreducible} as a Galois module, then $\mu =0$.  
	\end{conj}

        In this paper, we prove:

        \begin{thm}
        \label{thm: main-thm}
            Let $E/\Q$ be an elliptic curve. For all but finitely many primes $p$ with good ordinary reduction, we have $\mu \leq 1$. 
        \end{thm}
 
        \subsection{Strategy}

    The Iwasawa main conjecture says that to prove that $\mu \leq 1$, it suffices to show that the $p$-adic $L$-function of $E$ is not divisible by $p^2$. To show this, we express the coefficients of the $p$-adic $L$-function of $E$ in terms of modular symbols. We study these modular symbols using continued fractions.

        \subsection{Related Work} 

       Greenberg has also made a more general conjecture that also applies in the case where $E[p]$ is \textit{reducible}: if $E$ is an elliptic curve over $\Q$ and $p \neq 2$ is a good ordinary prime, there exists an isogenous curve $E'$ for which $\mu(E')=0$. Drinen \cite{drinen} has shown that this more general conjecture is false over general number fields, so it applies particularly for curves $E$ over $\Q$.  
        
	Greenberg's conjecture bifurcates into two cases: the case where $E[p]$ is reducible as a Galois module and second, there is the case where $E[p]$ is irreducible as a Galois module. In the \textit{reducible} case, we have the following theorem \cite[Theorem 1.3]{greenberg-vatsal} of Greenberg and Vatsal: if $E[p]$ has a Galois submodule which is either unramified at $p$ and odd, or ramified at $p$ and even, then $\mu =0$. Trifkovic \cite{mak} furthered this result by showing that Greenberg's conjecture holds for infinitely many curves $E$ with $E[p]$ reducible for the primes $p = 3$ or $5$. 

    In a different direction, Emerton-Pollack-Weston \cite{emerton-pollack-weston} showed that "$\mu=0$" is invariant in Hida families. That is, they showed that if a Hida family of modular forms contains a single form with $\mu=0$, then \textit{every} form in the Hida family has $\mu=0$. In addition, Ray recently \cite{ray2} showed that there is a positive proportion of elliptic curves $E/\Q$ for which $\mu_5(E) = \lambda_5(E)=0$. To prove this, Ray uses results of Bhargava-Shankar \cite{bhargava-shankar} on average ranks of Selmer groups. In another paper, Ray \cite{ray1} also showed that when $E[p]$ is irreducible, if the classical Iwasawa $\mu$-invariant vanishes for the splitting field $\Q(E[p])$, and a purely Galois theoretic condition on $E[p]$ holds, Greenberg's conjecture holds for $E$.

    In the supersingular case, Gajek-Leonard \cite{rylan} has recently proven upper bounds for the $\mu^{\pm}$ invariants of elliptic curves for all sufficiently large supersingular primes given certain conditions on the $\lambda^{\pm}$-invariants of the curve. In particular, Gajek-Leonard shows that if $E$ has rank zero then $\mu ^- \leq 1$ and $\mu^+ \leq 2$ hold on density $1$ set of good supersingular primes.  
    
    \subsection{Acknowledgements} I would like to thank Prof. Kumar Murty for his constant encouragement and mentorship. I would also like to sincerely thank Debanjana Kundu for guiding me in my first steps in research, for suggesting this problem to me, and for her mentorship! I'd like to thank Barry Mazur, Antonio Lei, and Rylan Gajek-Leonard for their helpful comments. I especially thank A. Lei and R. Gajek-Leonard for catching a mistake in an earlier version of this article. I also thank the members of the GANITA lab for listening to a version of the proof and offering helpful suggestions. 

        \section{Preliminaries}
    
    In this section, let $E/\Q$ be an elliptic curve with good ordinary reduction at an odd prime $p$. Let $\Q^{\text{cyc}}/\Q$ denote the cyclotomic $\Z_p$ extension. 
	
	We summarize the Iwasawa main conjecture for elliptic curves. First we will discuss the objects on the algebraic side of the main conjecture, then we will discuss the objects on the analytic side, and then state the Iwasawa main conjecture which bridges these two worlds.
	
	\subsection*{The algebraic side}

	If $K/\Q$ is an algebraic extension, the Selmer group of $E$ over $K$ is a subgroup of $H^1(K, E(\overline{\Q})_{\text{tors}})$ which fits into the fundamental exact sequence
	\[ 0 \to E(K) \otimes \Q/\Z \to \text{Sel}(E/K) \to \Sha(E/K) \to 0 \] 
	where $\Sha(E/K)$ denotes the Tate-Shafarevich group of $E$ over $K$. Now we put $K = \Q^{\text{cyc}}$. Then we can consider the Selmer group $\text{Sel}(E/\Q^{\text{cyc}})$. This Selmer group has an action of the Galois group $\Gamma=\Gal(\Q^{\text{cyc}}/\Q)$. We view the $p$-primary subgroup  $\text{Sel}(E/\Q^{\text{cyc}})_p$ as a $\Lambda$-module, where $\Lambda = \Zpt$. This ring $\Lambda$ is called the \textit{Iwasawa algebra}. By a deep theorem of Kato \cite{kato}, the Pontryagin dual $X(E/\Q^{\text{cyc}}) = \text{Sel}(E/\Q^{\text{cyc}})_p^{\vee}$ is finitely generated and \textit{torsion} as a $\Lambda$-module. By the structure theorem for finitely generated $\Lambda$-modules (see \cite[Theorem 13.12]{washington}), there is a pseudo-isomorphism
	\[ X(E/\Q^{\text{cyc}}) \sim \left( \oplus_{i=1}^n \Lambda/(f_i(T)^{a_i}) \right) \oplus \left( \oplus_{j=1}^m \Lambda/(p^{\mu_j}) \right),  \] 
	where the $f_i(T)$'s are irreducible distinguished polynomials in $\Lambda$. We define the (algebraic) Iwasawa invariants by 
	\[ \lambda_{E}^{\text{alg}} = \sum_{i=1}^n a_i \text{deg}(f_i(T)), \hspace{20pt}\text{and}\hspace{20pt} \mu^{\mathrm{alg}}_E = \sum_{j=0}^m \mu_j. \]
	The \textit{characteristic ideal} of $X(E/\Q^{\text{cyc}})$ is the ideal of $\Lambda$ generated by $p^{\mu}f_1(T)^{a_1}\dots f_n(T)^{a_n}$. 
	
	\subsection*{The analytic side}
	Let $\chi$ be an even Dirichlet character. Let $L(E, \chi, s)$ denote the Hasse Weil $L$-function of $E$ twisted by the character $\chi$. Let $\Omega_E$ denote the N\'eron period of $E$. By the work of Shimura (see \cite{shimura, shimura2}) the value $L(E, \chi, 1)/\Omega_E$ is an \textit{algebraic} number. 
	Mazur and Swinnerton-Dyer have attached to $E$ a $p$-adic $L$-function $\mathcal{L}_p(E, T) \in \Lambda \otimes \Q_p$ satisfying the following interpolation properties. Let $a_p = (p+1) - \# E (\F_p)$ and let $\a_p\in \Z_p^{\times}$ denote the unique $p$-adic unit root of the Hecke polynomial $X^2 - a_pX + p$. If $\chi$ is an even Dirichlet character of conductor $p^n$ and $p$-power order, then
    \begin{equation}
        \label{eqn: interpolation-property}
        \mathcal{L}_p(E, \chi(1+p)-1) = \dfrac{g(\chi)}{\a_p^{n+1}} \dfrac{L(E,\overline{\chi}, 1)}{\Omega_E}.
    \end{equation}
    Furthermore we also have
    \[ \mathcal{L}_p(E, 0) = \left( 1- \dfrac{1}{\a} \right)^2 \cdot \dfrac{L(E,1)}{\Omega_E}.
    \]
	Using the Weierstrauss preparation theorem \cite[Theorem 7.3]{washington}, we can factor the $p$-adic $L$-series as follows
	\[  \mathcal{L}_p(E, T) = p^{\mu^{\mathrm{an}}_E} \cdot u(T) \cdot f(T) \]
	where $f(T)$ is a distinguished polynomial of degree $\lambda_E^{an}$ and $u(T)$ is a unit in $\Lambda$. The values $\mu^{\mathrm{an}}_E$ and $\lambda_E^{an}$ are the (analytic) Iwasawa invariants of $E$.
	By \cite[Proposition 3.7]{greenberg-vatsal}, if $E[p]$ is irreducible as a Galois module, then $\mathcal{L}_p(E, T) \in \Lambda$. In other words, $\mu^{\mathrm{an}}_E \geqq 0$. 
	
	\subsection*{Iwasawa Main Conjecture}
	The Iwasawa \textit{main conjecture} relates the Selmer group of $E$ over $\Qcyc$ to the $p$-adic $L$-function of $E$. Precisely, suppose that $E[p]$ is irreducible as a Galois module. The \textit{main conjecture} asserts that the characteristic ideal of the $p$-primary Selmer group $X(E/\Q^{\text{cyc}})$ is generated by the $p$-adic $L$-function $\mathcal{L}_p(E, T)$. In particular, the main conjecture asserts that the algebraic and analytic Iwasawa invariants agree: $\mu_{E}^{\text{alg}}=\mu_{E}^{\text{an}}$ and $\lambda_{E}^{\text{alg}}=\lambda_{E}^{\text{an}}$. Kato \cite{kato} has proven one divisibility of the main conjecture, which we state as follows:
	
	\begin{thm}[Kato]
		\label{kato}
		Let $E/\Q$ be an elliptic curve, and let $p \geqq 5$ be a prime of good ordinary reduction. Assume that the mod $p$ Galois representation $\overline{\rho_{E,p}}: \Gal(\overline{\Q}/\Q) \to \text{GL}_2(\F_p)$ is surjective. Then the characteristic ideal of $X(E/\Q^{\text{cyc}})$ divides the ideal generated by $\mathcal{L}_p(E, T)$ in $\Lambda$. In particular we have $\mu_{E}^{\text{alg}}\leqq \mu_{E}^{\text{an}}$ and $\lambda_{E}^{\text{alg}}\leqq \lambda_{E}^{\text{an}}$.
	\end{thm} 

        In the CM case, the full main conjecture was proved by Rubin \cite[Theorem 12.3]{rubin-IMC} as a consequence of the Iwasawa main conjecture for imaginary quadratic fields:

        \begin{thm}[Rubin]
		\label{rubin}
		Let $E/\Q$ be an elliptic curve with CM, and let $p$ be an odd prime of good ordinary reduction. The main conjecture is true for $E$. In particular, $\mu_{E}^{\text{alg}} = \mu_{E}^{\text{an}}$ and $\lambda_{E}^{\text{alg}}=\lambda_{E}^{\text{an}}$.
	\end{thm} 

    In a beautiful paper, Skinner and Urban \cite[Theorem 1]{skinner-urban} proved the converse divisibility of the main conjecture (under mild conditions). Combined with Kato's theorem, this yields the full main conjecture under mild conditions:

        \begin{thm*}
            Let $E/\Q$ be an elliptic curve, and let $p \geqq 5$ be a prime of good ordinary reduction. Suppose that 
            \begin{enumerate}
                \item the mod $p$ Galois representation $\overline{\rho_{E,p}}: \Gal(\overline{\Q}/\Q) \to \text{GL}_2(\F_p)$ is surjective, and 
                \item there exists a prime $q \neq p$ such that $q$ strictly divides the conductor of $E$ and $\overline{\rho_{E,p}}$ is ramified at $q$.
            \end{enumerate} 
            Then the full main conjecture is true for $(E, p)$.
        \end{thm*}

        In a recent preprint, Burungale-Castella-Skinner \cite[Theorem 1.1.2]{burungale-castella-skinner} have removed the ramification assumption (2) in the Theorem above.

        \section{A sufficient condition for $\mu \leq 1$}

        \begin{dfn}
		Let $E/\Q$ be an elliptic curve and $f(z)$ be the modular form attached to $E$.
		Given a rational number $r$, define the \emph{modular symbol} $[r] \in \Q$ as follows
		\[ [r] = \dfrac{-\pi i}{\Omega_E} \left( \int_{r}^{i\infty} f(z)\,dz +  \int_{-r}^{i\infty} f(z)\,dz \right). \]	
		Here $\Omega_E$ denotes the Neron period of $E$.
	\end{dfn}

    The value $[r]$ depends on $E$ but we supress this from the notation. 

    \begin{lem}
        \label{lem: L-value-denominator}
        Let $E/\Q$ be an elliptic curve and let $r$ be a rational number.
        \begin{enumerate}
            \item The modular symbol $[r]$ is a rational number.

            \item If $p$ is an odd prime of good reduction and $E[p]$ is irreducible, then the denominator of $[r]$ is coprime to $p$. 

            \item The modular symbol $[r]$ belongs to $\delta_E^{-1} \Z$, where $\delta_E$ is a positive integer independent of $r$.
        \end{enumerate}
            
    \end{lem} 

    \begin{proof}
        The first two claims are well-known theorems of Manin \cite{manin} and Drinfeld \cite{drinfeld}. The last claim is also well-known (see for example \cite[Lemma 1.2 (i)]{mazur-rubin} for a proof).  
    \end{proof}

    Let $\mu_{p-1}$ denote the $(p-1)$-st roots of unity in $ \Z_p^{\times}$. If $\a \in \Z_p^{\times}$ is a $p$-adic unit, then let $[\a / p^n]$ denote $[a/p^n]$ where $0 < a < p^n$ is integer satisfying $a \equiv \a$ mod $p^n$.     
    
    Here is the main result of this section:

        \begin{prop}
            \label{prop: sufficient-conditions}
            Let $E/\Q$ be an elliptic curve and let $p$ be an odd prime of good ordinary reduction. Assume that $E[p]$ is irreducible. Furthermore if $E$ is not CM assume that the mod $p$ Galois representation $\overline{\rho_{E,p}}: G_{\Q} \to \text{GL}_2(\F_p)$ is surjective. Let $r$ be a non-negative integer. 
            
            Suppose that there exists a positive integer $n$ and a residue class $a$ modulo $p^{n}$ such that 
            \begin{equation}
            \label{eqn: suff-cond}
                \sum_{\eta \in \mu_{p-1}} a_p \left[ \dfrac{\eta a}{p^{n}} \right] -  \left[ \dfrac{\eta a}{p^{n-1}} \right] \not\equiv 0 \mod p^{r+1}.
            \end{equation}
            Then $\mu^{\text{alg}}_E \leq r$. 
        \end{prop}

        \begin{proof}
        By Kato's Theorem in the non-CM case or Rubin's Theorem in the CM case, to show that $\mu^{\text{alg}}_E \leq r$, it suffices to show that $p^{r+1}$ does not divide the $p$-adic $L$-function $\mathcal{L}_p(E, T)$. After choosing a topological generator for $\Gal(\Qcyc/\Q)$, there is an isomorphism
        \[ \Zpt \simeq \varprojlim_n \Z_p [\Gal (\Q_n/\Q) ]. \] 
        We let $(\mathcal{L}_p^{(n)}(E))_{n \geqq 1}$ be the image of $\mathcal{L}_p(E, T)$ under this isomorphism. Thus for each $n \geqq 1$ the element $\mathcal{L}_p^{(n)}(E)$ belongs to the group ring $\Z_p [\Gal (\Q_n/\Q) ]$. 

        For an element $a \in \Z_p^{\times}$, let $\sigma_a$ denote the corresponding element of $\Gal(\Q(\z_{p^{\infty}})/\Q)$ and let $\gamma_a$ be the image of $\s_a$ in $\Gal(\Qcyc/\Q)$. By \cite[Chapter I, Section 10]{mazur-tate-teitelbaum}, for each $n \geqq 1$ we have 
        \begin{equation}
            \label{eqn: mtt}
        \mathcal{L}^{(n)}_p(E) = \sum_{\substack{1 \leq a < p^{n} \\ (a,p) = 1}} \left( \dfrac{1}{\a_p^{n}} \left[ \dfrac{a}{p^{n}} \right] - \dfrac{1}{\a_p^{n+1}} \left[ \dfrac{a}{p^{n-1}} \right] \right) \cdot \gamma_a, 
        \end{equation}
        where $\a_p$ is the unit eigenvalue of Frobenius at $p$.  

        The $p$-adic $L$-function $\mathcal{L}_p(E, T)$ is nonzero mod $p^{r+1}$ if and only if there exists a positive integer $n$ such that $\mathcal{L}^{(n)}_p(E)$ is nonzero mod $p^{r+1}$. This is true precisely if there exists a positive $n$ and a residue class $a$ modulo $p^n$ such that 
        \begin{align}
            \label{eqn1}
            \sum_{\eta \in \mu_{p-1}} \dfrac{1}{\a_p^{n}} \left[ \dfrac{\eta a}{p^{n}} \right] - \dfrac{1}{\a_p^{n+1}} \left[ \dfrac{\eta a}{p^{n-1}} \right] \not\equiv 0 \mod p^{r+1}.
        \end{align}
        (Note that if $a, a' \in \Z_p^{\times}$, we have $\gamma_a = \gamma_{a'}$ if and only if $a = \eta a'$ for some $\eta \in \mu_{p-1}$.) Multiplying (\ref{eqn1}) through by $\a_p^{n+1}$, we conclude that if there exists $n$ and $a$ such that  
        \[  \sum_{\eta \in \mu_{p-1}} \a_p \left[ \dfrac{\eta a}{p^{n}} \right] -  \left[ \dfrac{\eta a}{p^{n-1}} \right] \not\equiv 0 \mod p^{r+1} \]
        then $\mu_E^{\text{alg}} \leq r$. Since $E[p]$ is irreducible and $p$ is an odd prime of good reduction, the modular symbols of $E$ all have denominator coprime to $p$ by Lemma~\ref{lem: L-value-denominator}. Since $\a_p \equiv a_p$ mod $p$, we can replace $\a_p$ with $a_p$ in the above formula. This proves Proposition~\ref{prop: sufficient-conditions}.
        \end{proof}

        By Serre's open image theorem, if $E$ is non-CM, the mod $p$ Galois representation of $E$ is surjective for all but finitely many primes. If $E$ is CM, it is known (for example by Mazur's theorem on isogenies of prime degree) that $E[p]$ is irreducible for all but finitely many primes. Therefore by Proposition~\ref{prop: sufficient-conditions}, to prove Theorem~\ref{thm: main-thm} it suffices to show:

        \begin{thm}
            \label{thm: actual-main-thm}
            Let $E/\Q$ be an elliptic curve. Then for all but finitely many primes $p$, there exists a positive integer $n$ and a residue class $a$ modulo $p^{n}$ such that 
            \[ \sum_{\eta \in \mu_{p-1}} a_p \left[ \dfrac{\eta a}{p^{n}} \right] -  \left[ \dfrac{\eta a}{p^{n-1}} \right] \not\equiv 0 \mod p^{2}. \]
        \end{thm}

        The rest of the paper is devoted to proving Theorem~\ref{thm: actual-main-thm}.     

        \section{Continued Fractions}

        Here is the main result of this section:

    \begin{prop}
        \label{prop: period-integrals-estimate}
        Let $E/\Q$ be an elliptic curve of conductor $N$ and let $r$ be a rational number. Then 
        \[ [r] \ll_N \log (\mathrm{denominator}(r)) + O_N(1). \]
    \end{prop}

    To prove this, we need some background from continued fractions. 

        \subsection{Recollections from the theory of continued fractions} A good reference is the introduction of \cite{goswami}. Any real number $x \in (0, \infty)$ has a continued fraction expansion defined as follows. Put $x_0=x$ and let $a_0 = [x_0]$, where $[\cdot]$ denotes the integer part (i.e: the `floor' function). If $x_1 = x_0 - [x_0]=0$, the process terminates there and $x=a_0$. If not, then we let $a_1 = \left[ \dfrac{1}{x_1} \right]$. Again, if $x_2 = \dfrac{1}{x_1}-\left[ \dfrac{1}{x_1} \right]=0$, the process stops and we have $x=a_0 + \dfrac{1}{a_1}$; if not, we let $a_2 = \left[ \dfrac{1}{x_2} \right]$, and so on. It is easy to see that if $x$ is a positive rational, then the process terminates after a finite number of steps, and in that case we can write $x$ as a continued fraction:
	
	\[ x= a_0 + \dfrac{1}{a_1 + \dfrac{1}{a_2 + \dfrac{1}{\dots + \dfrac{1}{a_n}}}} = [a_0; a_1, a_2, \dots, a_n].\]
	
	Writing $x_k = [a_0; a_1, \dots, a_k]$ for $k \leqq n$, we call $x_k$ the \emph{$k^{th}$ convergent} of $x$. 

        \begin{dfn}
        For a rational number $r$, let $\ell(r)$ denote the length of the continued fraction expansion of $r$; i.e: the number of convergents of $r$. For example, we have $2/3 = [0; 1,2]$, so $\ell(2/3) = 3$. 
        \end{dfn}

        \subsection{The Euclidean Algorithm} The Euclidean Algorithm takes two positive integers $a,b$ and calculates their GCD. The process is demonstrated in the example below, where $a = 56, b = 154$:    
        \begin{align*}
            154 &= \mathbf{2} \cdot 56 + 42 \\
            56 &= \mathbf{1} \cdot 42 + 14 \\
            42 &= \mathbf{3} \cdot 14,
        \end{align*}
        so $\mathrm{gcd}(a,b) = 14$ (= the right-most number in the last equation). The bolded numbers on the RHS are exactly the continued fraction coefficients of $a/b$, so 
        \[ 56/154 = [0;\mathbf{2},\mathbf{1},\mathbf{3}]. \]  
        In particular, for any rational number $a/b$, we have $\ell(a/b) = N + 1$, where $N$ is the number of steps in the Euclidean algorithm for the input $(a,b)$. 
        
        It is a fundamental theorem of Gabriel Lam\'e (see \cite[Page 54]{honsberger}) that the number of steps in the Euclidean algorithm for the input $(a,b)$ (where $0 \leqq a \leqq b$) is $\leqq 5 \log_{10}(b)$. So for any rational number $a/b$ with $0 \leqq a \leqq b$, 
        \[ \ell(a/b) \leqq 1 + 5 \log_{10} b \leqq 1 + \dfrac{5}{2} \log b, \]
        where in the last inequality $\log b$ denotes the usual base-$e$ logarithm. We have shown: 

    \begin{lem}
        \label{lem: continued-fractions}
        If $r$ is a  rational number, then 
        \[ \ell(r)  \leqq 1 + \dfrac{5}{2} \log (\mathrm{denominator}(r)).  \]
    \end{lem}

	\begin{proof}[Proof of Proposition~\ref{prop: period-integrals-estimate}]

    Recall that $[r]$ is defined as
    \[ [r] = \dfrac{\pi i}{\Omega_E} \left( \int_{r}^{i\infty} + \int_{-r}^{i\infty} f(z) \, dz \right). \]
    So to prove the Proposition, it suffices to show that for rational numbers $r$, 
    \[ \int_{r}^{i\infty} f(z) \, dz \ll_N \log(\mathrm{denominator}(r)) + O_N(1).  \]
    To do this, we use an observation of Manin that continued fractions are a valuable tool in studying modular symbols \cite[Proposition 1.6]{manin}. Given a rational number $r$, consider the successive convergents of $r$ in lowest terms:
	\[ \dfrac{b_{-2}}{a_{-2}} = \dfrac{0}{1}, \dfrac{b_{-1}}{a_{-1}} = \dfrac{1}{0}, \dfrac{b_0}{a_0} = \dfrac{b_0}{1}, \, \dots, \, \dfrac{b_{n-1}}{a_{n-1}}, \dfrac{b_n}{a_n}=\dfrac{b}{a},  \]
	where the first two are included formally. Then we can write:
	\begin{equation}
		\label{integral0}
		\int_{i\infty}^{r} f(z) \, dz = \int_{b_{-2}/a_{-2}}^{b_{-1}/a_{-1}} + 
		\int_{b_{-1}/a_{-1}}^{b_{0}/a_{0}} + \dots + \int_{b_{n-1}/a_{n-1}}^{b_{n}/a_{n}} f(z) \, dz.
	\end{equation}
	Now comes the key step: by \cite[Proof of Proposition 3.11]{stein} since the numbers $b_i/a_i$ are convergents, they satisfy $b_ka_{k-1} - b_{k-1}a_k = (-1)^{k-1}$ so that
	\begin{equation*}
		g_k = 
		\begin{pmatrix}
			b_k & (-1)^{k-1}b_{k-1} \\
			a_k & (-1)^{k-1}a_{k-1} 
		\end{pmatrix}
		\in SL_2(\Z).
	\end{equation*}  
	We can now write  $\frac{b_k}{a_k} = g_k(\infty) = g_{k+1}(0).$ Inserting this  into the integral (\ref{integral0}), we get
	\begin{equation}
		\label{eqn: integral}
		\int_{i\infty}^{r} f(z) \, dz =  \int_{g_{-1}(0)}^{g_{-1}(\infty)} + 
		\int_{g_0(0)}^{g_0(\infty)} + \dots + \int_{g_n(0)}^{g_n(\infty)} f(z) \, dz  
	\end{equation}
	Because $f(z) \, dz$ is a cusp form for $\Gamma_0(N)$, it follows that the differential form $f(z) \, dz$ is $\Gamma_0(N)$-invariant. Therefore, for any $\a, \b \in \mathbf{P}^1(\Q)$ and any $\gamma \in \Gamma_0(N)$, we have
	\[\int_{\a}^{\b} f(z) \, dz = \int_{\gamma \cdot \a}^{\gamma \cdot \b} f(z) \, dz.  \]
	This implies that the $k$-th integral $\int_{g_{k}(0)}^{g_{k}(\infty)}$ in (\ref{eqn: integral}) only depends on the \textit{equivalence class} of $g_k$ in $SL_2(\Z) / \Gamma_0(N)$. So for any coset $U \in SL_2(\Z) / \Gamma_0(N)$, if we let $c_U$ denote the number of matrices $g_k$ belonging to the coset $U$
	\[ c_U(r) = \#\{i: g_i \in U\},  \] 
	then we can rewrite equation (\ref{eqn: integral}) as follows:
	\begin{equation}
		\label{eqn: integral2}
		\int_{i\infty}^{r} f(z) \, dz = \sum_{U \in SL_2(\Z) / \Gamma_0(N)} c_U(r) \int_{U(0)}^{U(\infty)} f(z) \, dz. 
	\end{equation}
	On the RHS, the notations $U(0)$ and $U(\infty)$ are independent of the choice of representative of $U$. For all $U$ we trivially have $c_U(r) \leqq \ell(r) + 2$. By Lemma~\ref{lem: continued-fractions}, we have for all $U$:
    \begin{equation}
        \label{eqn: trivial-bound}
        c_U(r) \leqq 3 + \dfrac{5}{2} \log (\mathrm{denominator}(r)).
    \end{equation}
    Therefore (\ref{eqn: integral2}) gives us 
    \[ \left| \int_{i\infty}^{r} f(z) \, dz \right| \leqq \left(3 + \frac{5}{2} \log (\mathrm{denominator}(r)) \right) \sum_{U}  \left| \int_{U(0)}^{U(\infty)} f(z) \, dz \right|. \]
    Pick a constant $C = C(N)$ depending only on $N$ (and not on $r$) such that $\left| \int_{U(0)}^{U(\infty)} f(z) \, dz \right| \leqq C$ for all $U$. Then 
        \begin{align*}
            \left| \int_{i\infty}^{r} f(z) \, dz \right|
            &\leqq C \left(3 + \frac{5}{2} \log (\mathrm{denominator}(r)) \right) [SL_2(\Z): \Gamma_0(N)] \\
            &\ll_N \log (\mathrm{denominator}(r)) + O_N(1).
        \end{align*}
    This completes the proof of Proposition~
    \ref{prop: period-integrals-estimate}.  
        \end{proof}

    \section{A Proposition}

    Here is the main result of this section.
        
        \begin{prop}
            \label{prop: chinta}
            Let $E/\Q$ be an elliptic curve of conductor $N$ and let $p$ be an odd prime of good ordinary reduction. There exists a positive integer $n_0$ depending only on $N$ (but not on $p$) such that for all $n \geqq n_0$, there exists a residue class $a$ modulo $p^n$ such that   
            \[ \sum_{\eta \in \mu_{p-1}} a_p\left[ \dfrac{\eta a}{p^n} \right] -  \left[ \dfrac{\eta a}{p^{n-1}} \right] \neq 0. \]  
        \end{prop}

    To prove Proposition~\ref{prop: chinta}, we need two intermediate results. 
    
    \begin{lem}
                \label{sublem}
                 Let $\chi$ be an even primitive character of conductor $p^n$. Then 
                 \begin{enumerate}
                     \item We have 
                     \[ \sum_{a \text{ mod } p^n} \chi(a)  \left[ \dfrac{a}{p^n} \right] = g(\chi) \dfrac{L(E, \overline{\chi}, 1)}{\Omega_E}, \]
                     where $g(\chi)$ is the Gauss sum of $\chi$.
                     \item We have 
                     \[ \sum_{a \text{ mod } p^n} \chi(a)  \left[ \dfrac{a}{p^{n-1}} \right] = 0. \]
                 \end{enumerate}
            \end{lem}

            \begin{proof}
                The first claim follows from Birch's formula (see \cite[Formula (8.6)]{mazur-tate-teitelbaum}):
                \[ g(\chi) \, L(E, \overline{\chi}, 1) = \sum_{a \text{ mod }p^n} \chi(a) \left( \int_{a/p^n}^{i\infty} 2\pi i f(z) \, dz \right). \]

                To prove the second claim, for $m \in \Z$, define the extended Gauss sum of $\chi$ as
                \[ g(m, \chi) = \sum_{a \text{ mod }p^n} \chi(a) e^{2\pi i ma / p^n}. \]
    Since $\chi$ is primitive, by \cite[Theorem 9.5]{montgomery-vaughan} we have for all $m \in \Z$: 
                \begin{equation}
                \label{eqn: gauss-sum}
                    g(m, \chi) = \overline{\chi}(m) \cdot g(\chi).
                \end{equation}
                
                Write $f(z) = \sum_{m \geqq 1} a(m) q^m$. An easy calculation shows that
                \begin{align*}
                    \sum_{a \text{ mod }p^n} \chi(a) f \left(z + \frac{a}{p^{n-1}} \right) 
                    &=  \sum_{m \geqq 1} a(m) q^m \left( \sum_{a \text{ mod }p^n} \chi(a) e^{2\pi i ma/p^{n-1}} \right) \\
                    &= \sum_{m \geqq 1} a(m) q^m \cdot g(mp, \chi) \\
                    &= 0,
                \end{align*}  
                where the last equality follows from (\ref{eqn: gauss-sum}). The claim now follows from integrating the above equation along the vertical line from $0$ to $i\infty$.
            \end{proof}

        We need the following deep result of Chinta \cite[Theorem 2]{chinta}.

        \begin{thm}
            \label{thm: chinta}
            Let $E/\Q$ be an elliptic curve of conductor $N$. There is a constant $n_0$ depending only on $N$ such that for all primes $p$ not dividing $N$, if $\chi$ is a Dirichlet character of conductor $p^n$ with $n \geqq n_0$ we have $L(E, \chi, 1) \neq 0$.
        \end{thm}

         \begin{proof}[Proof of Proposition~\ref{prop: chinta}] 
            
            Let $\chi$ be a primitive character of $p$-power order and conductor $p^n$ with $n \geqq n_0$, where $n_0$ is as in Theorem~\ref{thm: chinta}. Since $\chi$ has $p$-power order and $p$ is odd, $\chi$ is even. We can then apply Lemma~\ref{sublem} to obtain  
            \[
                \sum_{a \text{ mod } p^n} \chi(a) \left( a_p \left[ \dfrac{a}{p^n} \right] -  \left[ \dfrac{a}{p^{n-1}} \right]  \right) 
                = a_p \, g(\overline{\chi}) \,  \dfrac{L(E, \chi, 1)}{\Omega_E} 
                \neq 0,
            \]
            where the $\neq 0$ follows from Theorem~\ref{thm: chinta}. Since $\chi$ has $p$-power order, $\chi$ is trivial on $\mu_{p-1} \sse \Z_p^{\times}$ when viewed as a map $\chi: \Z_p^{\times} \to \C^{\times}$. Therefore $\chi(a) = \chi(\eta a)$ for all $\eta \in \mu_{p-1}$, so we can rewrite the LHS of the above equation as  
            \[
                \sum_{\substack{a \text{ mod } p^n \\ a \equiv 1 \text{ mod } p}} \chi(a) \left( \sum_{\eta \in \mu_{p-1}} a_p \left[ \dfrac{\eta a}{p^n} \right] -  \left[ \dfrac{\eta a}{p^{n-1}} \right]  \right) 
                = a_p \, g(\overline{\chi}) \,  \dfrac{L(E, \chi, 1)}{\Omega_E} 
                \neq 0.
            \]
            Therefore for each $n \geqq n_0$ there exists a residue class $a$ modulo $p^n$ such that  
            \[  \sum_{\eta \in \mu_{p-1}} a_p \left[ \dfrac{\eta a}{p^n} \right] -  \left[ \dfrac{\eta a}{p^{n-1}} \right] \neq 0. \]
        \end{proof}
        
    \section{Proof of Theorem~\ref{thm: actual-main-thm}}

    We are now ready to prove Theorem~\ref{thm: actual-main-thm}. By Proposition~\ref{prop: period-integrals-estimate} we have 
    \begin{equation}
        \label{eqn: modular_symbol_estimate}
        [r] \ll_N \log (\mathrm{denominator}(r)) + O_N(1) 
    \end{equation}
    for all rational numbers $r$. By Proposition~\ref{prop: chinta}, we can pick a positive integer $n$ depending only on $N$ and a residue class $a$ modulo $p^{n}$ such that 
    \begin{equation}
    \label{eqn4}
        \sum_{\eta \in \mu_{p-1}} a_p \left[ \dfrac{\eta a}{p^{n}} \right] -  \left[ \dfrac{\eta a}{p^{n-1}} \right] \neq 0. 
    \end{equation}    
    By Lemma~\ref{lem: L-value-denominator}, there is a positive integer $\delta_E$ independent of $r$ such that the value
    \[ \delta_E \left( \sum_{\eta \in \mu_{p-1}} a_p \left[ \dfrac{\eta a}{p^{n}} \right] -  \left[ \dfrac{\eta a}{p^{n-1}} \right] \right) \]
    is a rational \textit{integer}. The Hasse-Weil Bound gives us $|a_p| \leqq 2p^{1/2}$. Combining this with (\ref{eqn: modular_symbol_estimate}) we obtain:
    \begin{align*}
        \left| \delta_E \left( \sum_{\eta \in \mu_{p-1}} a_p \left[ \dfrac{\eta a}{p^{n}} \right] -  \left[ \dfrac{\eta a}{p^{n-1}} \right] \right) \right| 
        &\ll_N (p-1) \cdot p^{1/2} \left(\log(p^n) + O_N(1) \right) \\
        &\ll_N p^{3/2} \left(\log(p) + O_N(1) \right) \\
        &\ll_N p^2
    \end{align*}
    where the second last inequality we use the fact that $n$ does not depend on $p$. Therefore for all $p$ sufficiently large we have 
    \begin{equation}
        \label{eqn: key-inequality}
        \left| \delta_E \left( \sum_{\eta \in \mu_{p-1}} a_p \left[ \dfrac{\eta a}{p^{n}} \right] -  \left[ \dfrac{\eta a}{p^{n-1}} \right] \right) \right| < p^2.
    \end{equation}
    Since the LHS of (\ref{eqn: key-inequality}) is a nonzero rational integer, the LHS of (\ref{eqn: key-inequality}) is nonzero mod $p^2$. If we additionally require that $p^2 > \delta_E$, then the quantity $ \sum_{\eta \in \mu_{p-1}} a_p \left[ \frac{\eta a}{p^{n}} \right] -  \left[ \frac{\eta a}{p^{n-1}} \right]$ is also nonzero mod $p^2$. 
    This completes the proof of Theorem~\ref{thm: actual-main-thm} and thus Theorem~\ref{thm: main-thm}.

	\bibliographystyle{amsalpha}
	\bibliography{bibliography}
	
\end{document}